\documentclass{amsart}
\usepackage{amssymb, latexsym, amsmath, amscd, epsfig, color}
\pdfoutput=1

\theoremstyle{plain}
\newtheorem{theo}{Theorem}[section]
\newtheorem{lem}[theo]{Lemma}
\newtheorem{cor}[theo]{Corollary}
\newtheorem{prop}[theo]{Proposition}
\newtheorem{defi}[theo]{Definition}

\newtheorem*{pro}{Proposition}
\theoremstyle{definition}

\newcommand{\R}{\ensuremath{\mathbb{R}}}

\newcommand{\X}{\ensuremath{\mathbb{X}^n}}
\newcommand{\N}{\ensuremath{\mathbb{N}}}
\newcommand{\Hyp}{\ensuremath{\mathbb{H}}^n}
\newcommand{\E}{\ensuremath{\mathbb{E}^n}}
\newcommand{\cG}{\ensuremath{\mathcal{G}}}

\newcommand{\cZ}{\ensuremath{\mathcal{Z}}}
\newcommand{\cP}{\ensuremath{\mathcal{P}}}

\newcommand{\G}{\Gamma}
\newcommand{\A}{\mathcal{A}}

\newcommand{\LL}{\mathcal{L}}

\newcommand{\Vol}{\ensuremath{\operatorname{vol}}}
\newcommand{\bs}{\backslash}
\newcommand{\s}{\sigma}
\newcommand{\g}{\gamma}

\newcommand{\tg}{\widetilde{\gamma}}

\newcommand{\tr}{\mathrm{tr}}

\newcommand{\ux}{\underline {x}}
\newcommand{\uy}{\underline {y}}
\newcommand{\uz}{\underline {z}}

\newcommand{\dd}{\delta}
\newcommand{\e}{\varepsilon}
\newcommand{\vp}{\varphi}

\begin{document}

\title[Volume entropy of hyperbolic buildings]{Volume entropy of hyperbolic buildings}
\author{Fran\c cois Ledrappier}
 \address{Department of Mathematics,  255 Hurley Hall, University of Notre Dame, Notre Dame IN} \email{ledrappier.1@nd.edu}
\author{Seonhee Lim}
\address{Department of Mathematics, 310 Malott Hall, Cornell University, Ithaca NY
14853-4201} \email{slim@math.cornell.edu}

\date{July 4, 2009}

\begin{abstract} We characterize the volume entropy of a regular building as the topological pressure of the geodesic flow on an apartment. We show that the entropy maximizing measure is not Liouville measure for any regular hyperbolic building. As a consequence, we obtain a strict lower bound on the volume entropy in terms of the branching numbers and the volume of the boundary polyhedrons.

\end{abstract}

\maketitle
\section{Introduction}
The {\it volume entropy} of a Riemannian manifold $(X, g)$ is
defined as the exponential growth rate of volume of balls in the universal cover:
\[ h(g)=\underset{r \to \infty}{\lim} \frac{\ln(\Vol_{g}(B_{g}(x, r)))}{r},\]
where $x \in X$ is a basepoint in the universal cover $\widetilde{X}$ of $X$, and $B_{g}(x, r)$ is the
$\widetilde{g}$-metric ball of radius $r$ centered at $x$ in $\widetilde{X}$.

The volume entropy has been extensively studied for closed Riemannian manifolds. This seemingly coarse asymptotic invariant carries a lot of geometric informations: it is related to the growth type of the fundamental group $\pi_1(M)$ (\cite{Mil}), the Gromov's simplicial volume (\cite{Gro}),
 the bottom of the spectrum of
Laplacian (\cite{Led}), and the Cheeger isoperimetric constant (\cite{Bro}). If the space $(X,g)$ is compact and non-positively curved, the volume entropy is equal to the topological entropy of the geodesic flow (\cite{Man} for manifolds, \cite{Leu1} for buildings) as well as to the critical exponent of the fundamental group (for example, see \cite{Pic}). 

In this paper, we are interested in the volume entropy of buildings. 
Our initial motivation to study the volume entropy of buildings comes from the fact that classical Bruhat-Tits buildings are analogues of symmetric spaces for Lie groups over non-archimedean local fields. However, we consider hyperbolic buildings as well, which are Tits buildings but not Bruhat-Tits buildings. We consider Euclidean and hyperbolic buildings, which are unions of subcomplexes, called {\it apartments}, which are hyperbolic or Euclidean spaces tiled by a Coxeter polyhedron. Euclidean buildings include all classical Bruhat-Tits buildings. Hyperbolic buildings, especially their boundary or properties such as quasi-isometry rigidity or conformal dimension have been studied by Bourdon, Pajot, Paulin, Xie and others (\cite{Bou}, \cite{DO}, \cite{BouPaj}, \cite{HP}, \cite{Leu2}, \cite{Xie}, \cite{Vo}). Volume entropy of hyperbolic buildings has been studied by Leuzinger, Hersonsky and Paulin (\cite{Leu1}, \cite{HeP}).

We first characterize the volume entropy of a compact quotient $X$ of a regular building as the topological pressure of some function on a quotient of an apartment, which is roughly the exponential growth rate of the number of longer and longer geodesic segments (which are separated enough) in one apartment, counted with some weight function (see Theorem~\ref{th:pressure}). 
The dynamics of the geodesic flow of one apartment of a building is better understood than that of the building, which makes this characterization useful for the later parts of the paper.


There are two naturally defined measures on the boundary of the universal cover of a closed Riemannian manifold of negative curvature, which are the visibility measure and the Patterson-Sullivan measure. They correspond to invariant measures of the space of geodesics, namely the Liouville measure and  the Bowen-Margulis measure, respectively. 
Liouville measure can be thought of locally as the product of the volume form on the manifold and the canonical angular form on the unit tangent space. Bowen-Margulis measure is the measure which attains the maximum of measure-theoretic entropy, and it can be thought of as the limit of average of all the Lebesgue measures supported on longer and longer closed geodesics in the given closed manifold.  Katok made a conjecture that Liouville measure and Bowen-Margulis measure coincide if, and only if, the metric is locally symmetric and he showed it for surfaces (\cite{Kat}).

 
Bowen-Margulis measure associated to  a compact quotient of a building $\Delta$ is defined as follows. Since $\Delta$ is a CAT(-1) metric space, there is a unique Patterson-Sullivan measure on the boundary of $\Delta$ and Sullivan's construction yields a unique geodesic flow invariant probability measure $m_{BM}$ on the space of geodesics of $X$. (Here by a geodesic, we mean an isometry from $\R$ to $X$, i.e. a marked geodesic.) This measure is ergodic and realizes the topological entropy (see \cite{Ro}). We call it the Bowen-Margulis measure.  Another family of measures invariant under the geodesic flow of the building is the family of measures proportional to  Liouville measure on the unit tangent bundle of each apartment of $\Delta$. We will say that a measure $\mu$ \textit{projects to Liouville measure} if $\mu$ projects to one of these measures.

Our main result is that Bowen-Margulis measure does not project to Liouville measure for any regular hyperbolic building. This result is unexpected for buildings of constant thickness starting from a regular right-angled hyperbolic polygon, since they have a very symmetric topological structure around links of vertices, and they are built of symmetric spaces for which Liouville measure coincides with Bowen--Margulis measure. In retrospect, one possible explanation is that buildings, even the most regular ones, correspond to manifolds of variable curvature rather than to locally symmetric spaces.

Remark that if we vary the metric on $X$ to a non-hyperbolic metric, Bowen--Margulis measure might project to Liouville measure. If this happens for the metric of minimal volume entropy, it would extend the result of R. Lyons (\cite{Lyo}) on combinatorial graphs which are not regular or bi-regular: in terms of metric graphs, Lyons constructed some examples of graphs, including all (bi-)regular graphs with the regular metric,   for which Liouville measure and Bowen-Margulis measure are in the same measure class. Interpreted with the characterization of entropy minimizing metric on graphs (\cite{Lim}) these examples of metric graphs minimize the volume entropy among all metric graphs with the same combinatorial graph. In \cite{Lim}, the second author showed that the metric minimizing volume entropy is determined by the valences of the vertices, and in particular, it is not ``locally symmetric" (i.e. not all edges have the same length) if the graph is not regular or bi-regular.

The characterization we use to show the main result enables us to compute explicitly the maximum of the entropies of the measures projecting to Liouville measure. Consequently, we obtain a  lower bound of volume entropy for compact quotients of a regular hyperbolic building in terms of purely combinatorial data of the building, namely the thickness, and the volume of the panels  of the quotient complex. 
\vspace{.1 in}

\noindent Now let us state our results more precisely.

\subsection{Statements of results}\label{sec:first}
\vspace{-.05 in}Let $P$ be a Coxeter polyhedron, i.e., a convex polyhedron, either in $\mathbb{H}^n$ or in $\mathbb{R}^n$, each of whose dihedral angle is of the form $\pi/m$ for some integer $m \geq 2$. Let $(W,S)$ be the \textit{Coxeter system} consisting of the set $S$ of reflections of $\X$ with respect to the faces of codimension $1$ of $P$, and the group $W$ of isometries of $\X$ generated by $S$. It has the following finite presentation:
\[ W= \left< s_i : s_i ^2=1, (s_i s_j)^{m_{ij}}=1 \right>,\]
where $m_{ii}=1$, $m_{ij} \in \N \cup\{\infty\}$. 
Let $\Delta$ be a hyperbolic or Euclidean regular building of type $(W, S)$, equipped with the symmetric metric (i.e. metric of constant curvature) induced from that of $P$. 

If $\Delta$ is a {\it right-angled} building (i.e. all the dihedral angles are $\pi/2$), for a given family of positive integers $\{ q_i \}$, there exists a unique building of type $(W,S)=(W(P),S(P))$ up to isometry such that the number of chambers adjacent to the $(n-1)$-dimensional face $F_i$, called the \textit{thickness of $F_i$}, is $q_i+1$ \cite{HP}. 

 The building $\Delta$ is equipped with a family of subcomplexes, called {\it apartments}, which are isometric to tessellations of $\mathbb{H}^n$ or $\mathbb{R}^n$ by $P$. For a fixed chamber $C$ and an apartment $\A$ containing it, there is a {\it retraction} map $\rho : \Delta \to \A$, whose restriction to each apartment containing $C$ is an isometry. 
 
Let $X = \Gamma \backslash \Delta$ be a compact quotient of $\Delta$, which is a polyhedral complex whose chambers are all isometric to $P$. We are interested in the volume entropy $h_{\Vol}(X)$ of $X$. It is easy to see that $h_{\Vol}(X)$ is positive if the building is thick (i.e. $q_i+1 \geq 3, \forall i$), as the entropy is bounded below by that of an embedded tree of degree at least $3$. 

Let us fix a fundamental domain $\widehat{X}$ in $\Delta$. Let us fix a chamber $C$ contained in $\widehat{X}$, an apartment $\A$ of $\Delta$ containing $C$, and a retraction map $\rho: \Delta \to \A$ centered at a chamber $C$. Since the Coxeter group $W$ is virtually torsion free (for example by Selberg's Lemma and Tit's theorem, see \cite{Dav} page 440), there is a finite index torsion-free subgroup $W'$ of $W$ such that $Y=W' \backslash \A$ is a compact quotient of an Euclidean or hyperbolic space (i.e. a manifold rather than a complex of groups).
Let $T^{1}(Y)$ be the unit tangent bundle of $Y$. The set of  $(n-1)$-dimensional faces of $\A$ is $W'$-invariant and projects to a totally geodesic subset $\LL$ of $Y$. In particular, the unit tangent bundle  $T^{1}\LL$ is  a finite union of closed codimension one subsets of $T^{1} Y$ which are invariant  under the geodesic flow.
Let $\tilde {v} $ be a vector in $T^{1} \A$ and $v$ its projection on $T^{1}Y$. We  define a  weight function $f$ such that the number $\int_0^T f(\varphi_s (v)) ds$ is approximatively the number of geodesic segments in $\Delta$ starting at $\tilde {v}$ and of length $T$. Let us first define two functions $q$ and $l$ on $T^1Y$. Let $v$ be in $T^{1}Y \setminus T^{1}\LL$ and denote $\gamma _v$ the  geodesic in $T^{1}(Y)$ with initial vector $v$. By abuse of notation, let us denote the lift of $\gamma_v$ in $\A$ embedded in the building $\Delta$ by $\tg _v$. Let $q(v)+1$ be the thickness of the $(n-1)$-dimensional face that the geodesic $\tg _v$ intersects last before or at time zero, and let $l(v)$ be the distance between two points of the faces that $\tg _v$ meets just before and after time zero. Define $q$ and $l$ analogously on $T^{1}\LL$ (see Definition~\ref{def:l} for a precise formulation).

\begin{theo}\label{th:pressure}  Let $X$ be a compact quotient of a regular Euclidean or hyperbolic building. Let $Y$ be a compact quotient of an apartment defined as above, and let $q$ and $l$ be the functions defined as above.  Denote by $h_{\mu}$ the measure-theoretic entropy  of a measure $\mu$ invariant under the geodesic flow (on the unit tangent bundle of $Y$). Then,
\[ h_{\Vol}(X)= \sup_{\mu} \left \{ h_\mu(\varphi) +  \int_{T^{1}Y} \frac{\ln q }{l} d\mu \right \}. \]
\end{theo}

\noindent \textit{Remark.} This theorem holds for more general hyperbolic metrics that we can consider on a ``metric" (regular) building (imagine varying the shape of each copy of $P$ for example), namely when the resulting compact quotient $X = W\backslash \Delta$ has a convex fundamental domain of $X$ in $\Delta$ which is contained in one apartment. However, in this paper, we restrict ourselves to the given metric on the building (which  is given by the metric on $P$).

\vspace{.1 in}
For the proof of Theorem \ref{th:pressure}, recall that by \cite{Leu1}, $h_{\Vol}(X)$  is given by the maximum  of the metric entropies $h_m$ of invariant probability measures for the geodesic flow on $\cG(X)$. We will associate to each invariant probability measure $m$ on $\cG(X)$  a measure $\tau(m)$ on $T^{1}Y$ which is invariant under the geodesic flow. Using the  relativized variational principle (\cite{LW}), we can show that,  for any ergodic measure $\mu $ on $T^{1} Y$
\[  \underset{\tau (m) = \mu }{\sup} \{ h_m \}  \; = \;  h_\mu(\varphi) +  \int_{T^{1}Y} \frac{\ln q }{l} d\mu. \]
Theorem \ref{th:pressure} follows.
Moreover, since the Bowen-Margulis measure $m_{BM}$ achieves the maximum of the entropy,  it follows from the relativized variational principle that the maximum in the formula from Theorem \ref{th:pressure} is achieved by $\mu = \tau(m_{BM})$. Another application of the formula follows from the computation of the integral when  the measure $\mu$ is the Liouville measure $\mu_L$ on $T^{1}Y$. We have:

\begin{prop}\label{prop:santalo}
Let $P$ be a convex polyhedron in $\X$, either hyperbolic or Euclidean. Let us denote the Liouville measure by $m_L$. Then
$$\int_{T^{1}(P)} \frac{\ln q}{l} dm_L = c_{n} \underset{F}{\sum} \ln q(F) \Vol(F),$$
where $c_n$ is the volume of the unit ball in $\mathbb{E}^n$ and
where the sum is over the set of $(n-1)$-dimensional faces of $P$.
In particular, if $q$ is a constant, then
$$\int \frac{\ln q}{l} dm_L = c_n \ln q \Vol(\partial P),$$
 where $\Vol(\partial P)$ is the $(n-1)$--volume of the boundary of $P$.
\end{prop}

\begin{cor}\label{cor:lowerbound} [Lower bound for entropy] The volume entropy $h_{\Vol}(\Delta)$ of a regular hyperbolic building $\Delta$ of type $(W(P),S(P))$  is bounded below by

\[ h_{\Vol} (\Delta)  \; \geq \; (n-1) + \frac{1}{\Vol(P)} \underset{F :\; \mathrm{face \; of\; P}}{\sum} \ln q(F) \Vol(F). \]
The volume entropy $h_{\Vol}(\Delta)$ of a regular Euclidean building $\Delta$ of type $(W(P),S(P))$  is bounded below by
\[ h_{\Vol} (\Delta)  \; \geq \; \frac{1}{\Vol(P)} \underset{F :\; \mathrm{face \; of\; P}}{\sum} \ln q(F) \Vol(F). \]

\end{cor}

Indeed, for the normalized Liouville measure $ \frac{1}{\Vol(T^1Y)}  m_L$, the entropy is $(n-1)$ for hyperbolic space ($0$ for Euclidean space, respectively) and the integral of $\frac{\ln q}{l} $ is $\frac{c_n}{\Vol(T^1Y)} \underset{F :\; \mathrm{face \; of\; P}}{\sum} \ln q(F) \Vol(F) $ times the number of $n$-dimensional faces in $Y$ (Proposition \ref{prop:santalo}). This number of $n$-dimensional faces in $Y$ is exactly $\frac {\Vol (Y)}{\Vol (P)} = \frac {\Vol (T^{1}Y)}{\Vol (T^{1}P)}$. Corollary \ref{cor:lowerbound} follows by reporting in the formula in Theorem \ref{th:pressure}. 

\

Now let $\Delta$ be a regular hyperbolic building. This includes for example Bourdon's buildings (a building $\Delta$ is called a \textit{ Bourdon's building} if $P$ is a regular hyperbolic right-angled polygon). Properties such as quasi-isometry rigidity or conformal dimension of Bourdon's buildings (and more generally Fuchsian buildings) have been studied by Bourdon, Pajot, Xie, and others (\cite{Bou}, \cite{BouPaj}, \cite{Xie}).
Using  the above theorem, we show that the entropy maximizing measure does not project to Liouville measure and obtain a strict lower bound as a consequence.
\begin{theo}\label{th:Liouville} Let $X$ be a compact quotient of a regular hyperbolic building of type $(W(P),S(P))$.
Then Bowen-Margulis measure does not project to the Liouville measure on $T^1(P)$. Consequently, the following strict inequality holds:
\[ h_{\Vol} (X) >(n- 1) + \frac{1}{\Vol (P)} \sum_{F} \ln q (F) \Vol(F), \]
where the sum is over all panels of the polyhedron $P$.
\end{theo}
 
 The proof of Theorem \ref{th:Liouville} amounts to showing that the Liouville measure on $T^1(Y)$ cannot realize the maximum because it cannot be an equilibrium measure for the function $\frac {\ln q}{l}$. One might be tempted to use criteria from thermodynamical formalism for this problem but the function $\frac{\ln q}{l} $ is neither continuous nor bounded. We will replace the function $\frac {\ln q}{l}$ by a function $f$ which is bounded and Lipschitz-continuous outside of the singular set $T^1(\LL)$. Since the measure $\tau (m_{BM})$ is ergodic, either it is supported by $T^1(\LL)$ and in that case, $\tau (m_{BM})$ cannot be the Liouville measure, or it is supported on the regular set. In dimension 2, there is a Markov coding for the regular set and we can use arguments from thermodynamical formalism on that invariant set. For the higher dimensional hyperbolic case, the geodesic flow is Anosov, thus admits abstract Markov codings (\cite{Bow1}, \cite{Rat}), but  the arguments are more delicate since there is no known explicit Markov coding adapted to the regular set.

The paper is organized as follows. After recalling necessary background, we define pressure of a measurable function and prove some of its properties in Section ~\ref{s:background}. In Section ~\ref{s:characterization}, We give some characterization of volume entropy of compact quotients of general regular buildings, namely Theorem~\ref{th:pressure} and another characterization which is analogous to that of graphs. In Section  ~\ref{s:lowerbound}, we show Proposition~\ref{prop:santalo}, and Theorem~\ref{th:Liouville} restricting ourselves to hyperbolic buildings. 


\section{Preliminaries}\label{s:background}




\subsection{Buildings} \label{ss:building}
In this section we recall definitions and basic properties of Euclidean and hyperbolic buildings. See \cite{GP} and
the references therein for details.

Let  $P$ be a Coxeter polyhedron in $\X$, where  $\X$ is $\Hyp$, $\mathbb{S}^n$ or $\E$ (with its standard metric of constant curvature $-1$, $1$  and $0$, respectively).  It is a compact, convex regular polyhedron each of whose dihedral angle is of the form $\pi/m$ for some integer $m \geq 2$. Let $(W,S)$ be the \textit{Coxeter system} consisting of the set $S$ of reflections of $\X$ with respect to the $(n-1)$-dimensional faces of $P$, and the group $W$ of isometries of $\X$ generated by $S$. It has the following finite presentation:
\[ W= < s_i : s_i ^2=1, (s_i s_j)^{m_{ij}}=1 >,\]
where $m_{ii}=1, m_{ij} \in \N \cup \{ \infty \}.$ 

A \textit{polyhedral complex $\Delta$ of type $(W,S)=(W(P),S(P))$} is a CW-complex such that there exists a morphism of
CW-complexes, called a \textit{function type}, $\tau : \Delta \to P$, for which its restriction to any
maximal cell is an isometry.
\begin{defi}{\rm [building]}\label{def:building}
Let $(W,S)$ be a Coxeter system of $\X$. A {\rm building} $\Delta$ of type $(W,S)$ is a polyhedral
complex of type $(W,S)$, equipped with a maximal family of subcomplexes, called {\rm apartments}, polyhedrally
isometric to the tessellation of $\X$ by $P$ under $W$, satisfying the following axioms:

\begin{enumerate}
\item{} for any two cells of $\Delta$, there is an apartment containing them,
\item{} for any two apartments $\A, \A'$, there exits a polyhedral isometry of $\A$ to $\A'$ fixing $\A \cap
\A'$.
\end{enumerate}
\end{defi}
A building is called {\it hyperbolic, spherical, and Euclidean (or affine)} if $\X$ is $\Hyp$, $\mathbb{S}^n$, $\E$, respectively.
The {\it link} of a vertex $x$ is a $(n-1)$-dimensional spherical building, whose vertices are the edges of $\Delta$ containing $x$,
and two vertices (two edges of $\Delta$) are connected by an edge if there is a $2$-dimensional cell containing both
edges of $\Delta$, etc. In dimension 2, the link of a vertex is a bipartite graph of diameter $m$ and girth $2m$, where $\pi/m$ is the dihedral angle at the vertex. 

The building $\Delta$ is a $CAT(\kappa)$-space, with $\kappa$ the curvature of $\X$ and its links are $CAT(1)$-spaces.

Cells of maximal dimension are called \textit{chambers}. Cells of dimension $(n-1)$ (i.e. interesections of two chambers) are called \textit{panels}.
For any panel $F$ of $P$, let $q(F)+1$ be its \textit{thickness}, i.e. the number of chambers containing it.
A building is called {\it regular} if the thickness depends only on the function type. A building is said to be {\it thick} if $q(F)+1 \ge 3$ for all $F$.

For a fixed chamber $C$ and an apartment $\A$ containing it, there exists a map $\rho: \Delta \to \A$, called the \textit{retraction map from $\Delta$ onto $\A$ centered at $C$}. It fixes $C$ pointwise, and its restriction to any apartment $\A'$ is an isometry fixing $\A \cap \A'$. 

\vspace{.1 in} \noindent {\textbf{Examples.}} Examples of Euclidean buildings of dimension $1$ include any infinite locally finite tree $(T,d)$ without terminal vertices with a combinatorial metric $d$. 
Products of locally finite trees are naturally Euclidean buildings.
Classical examples of buildings are Bruhat-Tits buildings, which are analogues of symmetric spaces for Lie groups over non-archimedean local fields. These are Euclidian buildings, and in dimension $2$, the polygon $P$ is either a triangle or a rectangle. 

By a theorem of Vinberg \cite{V}, the dimension of a hyperbolic building of type $(W,S)$ is at most 30.
Examples of hyperbolic buildings include Bourdon's buildings which are hyperbolic buildings of dimension $2$ with all its dihedral angles $\pi/2$. It implies that the link of any vertex is a  complete bipartite graph. 

More generally, a $2$-dimensional regular hyperbolic building is called a Fuchsian building. In dimension $\geq 3$, there exist uncountably many non-isomorphic hyperbolic buildings with some given polyhedron $P$ \cite{HP}. Hyperbolic buildings also appear in Kac-Moody buildings \cite{Rem}.
There are also right-angled non-hyperbolic buildings (\cite{Dav}). 

\subsection{Volume entropy and topological entropy} \label{sec:topentropy}
In this section, we recall the fact that the volume entropy is equal to the topological entropy of the geodesic flow for hyperbolic and Euclidean buildings  (\cite{Man}, \cite{Leu1}). We also recall the  Variational Principle, which will be used in Section~\ref{sec:pressure}.

Let $X$ be a compact quotient of a building $\Delta$ of type $(W(P), S(P))$. Let $h_{\Vol}(g)$ be the volume entropy of $(X,g)$ defined in the introduction, where $\Vol_{g}(S)$ of a subset $S \subset \Delta$ is the piecewise Riemannian volume, i.e.  the sum of the volume of $S \cap C$ for each chamber $C$ which is a polyhedron in $\X$.

 The entropy $h_{\Vol}(g)$ does not depend on the base point $x$. It satisfies the homogeneity
property $ h_{\Vol} (\alpha g) = \frac{1}{\sqrt{\alpha}} h_{\Vol} (g),$
 for every $\alpha>0$ if the dimension of $\Delta$ is at least two. It is easy to see that for any thick building, either hyperbolic or Euclidean, the volume
 entropy is positive, as it contains a tree of degree at least $3$ along a geodesic.
 
 \vspace{.1 in}
 \noindent Let us recall the topological entropy of the geodesic flow of $(X,g)$. The space $(X,g)$ is geodesically complete and locally uniquely geodesic. Let $\cG (X)$ be the set of all geodesics of $X$, i.e. the set of isometries from $\R$ to $X$. The geodesic flow $\vp_t$ on $\cG (X)$ is defined by $\g \mapsto \vp_t(\g)$, where $\vp_t(\g)(s)=\g(s+t).$ We define a metric $d_\cG$ on $\cG (X)$ by
\[ d_\cG(\g_1, \g_2)= \int_{-\infty}^{\infty} d(\g_1(t), \g_2(t)) \frac{e^{-|t|}}{2} dt. \]
The metric space $(\cG (X), d_\cG)$ is compact by Arzela-Ascoli theorem.
Define a family of new metrics on $\cG(X)$:
\[ d_{T} (x, y) = \underset{ 0 \leq t < T}{\max} d(\vp_t(x), \vp_t (y)). \]
Intuitively, it measures the distance between the orbit segments $\{ \vp^t (x)\; |\; 0 \leq t < T \}$. We say that a set $S \subset \cG(X)$ is
$(T, \e)$-\textit{separated} ($(T,\e)$-\textit{spanning}) if any two points in $S$ are of $d_{T}$-distance at least $\e$ (if any point of $\cG(X)$ is in $\e$-neighborhood of $S$ in the metric $d_{T}$, respectively). As $T$
increases or $\e$ decreases, the cardinality of maximal $(T, \e)$-separated (minimal $(T,\e)$-spanning) sets increase. Let us denote the maximal cardinality
of $(T,\e)$-separated set and the minimal cardinality of $(T,\e)$-spanning set by $N^{\Delta}_d(\vp, \e, T)$ and $S^{\Delta}_{d}(\vp, \e, T)$, respectively.

\begin{defi}\label{def:topent}{\rm [topological entropy]}
The topological entropy of the geodesic flow $\vp_t$ is defined as follows:
$$ h_{\mathrm{top}}(X, \vp_t) := \underset{\e \to 0}{\lim} \;\underset{ T \to \infty}{\limsup} \frac{1}{T} \ln N^{\Delta}_d(\vp, \e, T).$$
Equivalently,
$$ h_{\mathrm{top}}(X, \vp_t) := \underset{\e \to 0}{\lim} \;\underset{ T \to \infty}{\limsup} \frac{1}{T} \ln S^{\Delta}_d(\vp, \e, T).$$
\end{defi}

\begin{theo}{\rm[\cite{Leu1}]}\label{thm:Enrico}
 The volume entropy $h_{\Vol}(X)$ is equal to the topological entropy of the
geodesic flow for any Hadamard space $X$, in particular any compact quotient $X$ of a hyperbolic or Euclidian
building. 
\[h_{\Vol}(X) = h_{\mathrm{top}}(X, \vp_t).\]
\end{theo}

Let us denote the measure-theoretic entropy of a  $\vp _t$-invariant probability measure $m$ for the transformation $\vp _1$ by $h_m$ (see \cite{Wal} for definition of measure-theoretic entropy).
The following theorem gives us a fundamental relation between measure-theoretic entropy and topological entropy. 
\begin{theo}{\rm[Variational principle (\cite {Wal}, page 218)]}\label{thm:variational}
\[
 h_{\mathrm{top}}(X, \vp_t) \; = \; \underset{ m \in \mathcal{M}(\vp_{t}) }{\sup} \{ h_m  \} 
 \]
where $\mathcal{M}(\vp_t)$ is the set of all $\vp_t$-invariant probability measures on $\cG(X)$. 
\end{theo}

\subsection{Relativized Variational Principle}

Let $\pi : \cG (X) \to \cZ$ be a continuous map onto a compact space $\cZ$ and assume that there is a  one-parameter flow $\phi _t, t \in \R,$ of homeomorphisms of $\cZ$ such that $\pi \circ \vp _t = \phi _t \circ  \pi.$ For each point $ z\in \cZ$, $\pi ^{-1}(z) $ is a closed subset of $\cG (X).$  In the same way, for each $\phi $-invariant probability measure  $\mu$ on $\cZ$, there is a nonempty closed convex set, which we denote by $\mathcal{M}_\mu(\vp_t)$, of invariant probability measures $m$ on $\cG (X)$ such that $\pi _{\star} m  = \mu.$ The Relativized Variational Principle computes the maximum, over the convex set $\mathcal{M}_\mu(\vp_t)$, of the entropies $h_m $ in terms of the  entropy $h_{\mu}$ of the measure $\mu $ with respect to the transformation $\phi _1$ and the topological entropies of the non-invariant sets $\pi ^{-1} (z)$. The definition of the topological entropy of a non-invariant set is a direct extension of the definition of the topological entropy of the whole set:
let $A$ be a closed subset of $\cG (X)$ and let us denote the maximal cardinality
of a $(T,\e)$-separated subset of $A$ and the minimal cardinality of  a $(T,\e)$-spanning subset  of $A$ by $N^{\Delta}_d(\vp, \e, T, A)$ and $S^{\Delta}_{d}(\vp, \e, T, A)$, respectively.

\begin{defi}\label{def:topentsub}{\rm [topological entropy of a closed subset]}\cite{Bo3}
The topological entropy of the closed subset $A$ for the geodesic flow $\vp_t$ is defined as follows:
$$ h_{\mathrm{top}}(A, \vp_t) := \underset{\e \to 0}{\lim} \;\underset{ T \to \infty}{\limsup} \frac{1}{T} \ln N^{\Delta}_d(\vp, \e, T, A).$$
Equivalently,
$$ h_{\mathrm{top}}(A, \vp_t) := \underset{\e \to 0}{\lim} \;\underset{ T \to \infty}{\limsup} \frac{1}{T} \ln S^{\Delta}_d(\vp, \e, T, A).$$
\end{defi}

Observe that the topological entropy of a finite set is zero. In general, we have:
\begin{theo}{\rm [Relativized Variational Principle (\cite{LW})]} \label{thm:RVP}
With the above notations, for all $\phi$ invariant probability measure $\mu$ on $\cZ$, 
\[ \underset{ m \in \mathcal{M}_\mu(\vp_{t}) }{\sup} \{ h_m \} \; = \; h_{\mu} + \int _{\cZ} h_{\mathrm{top}} (\pi ^{-1} (z), \vp _t ) d\mu (z). \]
\end{theo}

\subsection{Pressure of a measurable function}

Let $Z$ be a measurable  space and  let $T$ be a measurable transformation of $Z$  ($\vp _t$ one-parameter group of measurable transformations of $Z$, respectively). Denote $\mathcal {M} (Z, T)$ ($\mathcal {M} (Z, \vp) $) the set of $T$-invariant ($\vp$-invariant, respectively) probability measures on $Z$. We assume that $\mathcal {M}(Z, T)$ ($\mathcal{M}(Z, \vp)$) is non empty.  Let $f: Z \to \R$ be a measurable bounded function. We define the pressure  $\cP_Z (f, T) $ ($\cP _Z (f, \vp )$, respectively) as follows.
\begin{defi}\label{def:pressure}{\rm [pressure of a measurable function]}
\begin{align*}
\cP_Z(f,T) \; &= \; \underset {m \in \mathcal {M} (Z,T )} \sup \{ h_m + \int f dm \}\\ ( \cP_Z(f, \vp) \;& = \; \underset {m \in \mathcal {M} (Z,\vp )} \sup \{ h_m + \int f dm \} , \textrm{respectively}).
\end{align*}
\end{defi}
In the case when the function $f$ is continuous, there are direct definitions of the pressure by putting suitable weights on the orbits in the definition \ref{def:topent}. The Variational Principle extends to the result that this topological quantity is indeed the pressure $\cP_Z(f,T)$. Let us denote it by $\cP(f)$ if there is no ambiguity. 
\begin{defi}\label{def:newpressure}{\rm [essentially cohomologous functions]} 
Two measurable functions   $f$  and $g$ are called {\rm essentially cohomologous}  if their integrals coincide with respect to any $\varphi_t$-invariant measure $\mu$:
\[ \int_{Z} f d\mu = \int_{Z} gd\mu. \]
\end{defi}
For such a pair of functions $f, g$, the pressure given by definition \ref{def:pressure} clearly coincide:
\[ \mathcal{P}(f) = \underset{\mu \in \mathcal{M}(\varphi_{t})}{\sup} \{ h_\mu + \int_{\cZ} f d\mu\} = \underset{\mu \in \mathcal{M}(\varphi_{t})}{\sup} \{ h_\mu + \int_{\cZ} g d\mu\} = \cP (g)\]
where the supremum is taken over all $\varphi_t$-invariant Borel probability measures.

This notion of pressure behaves handily with suspensions. Let $(Z,T)$ be a measurable transformation and $l: Z \to  \R$ a positive measurable function on $Z$, called the ceiling function, we define the suspension flow as follows:
\begin{enumerate}
\item Consider the space $\overline {Z} := Z \times [0, \infty] / \sim$, where $(x, l(x)) \sim (\sigma_A(x) , 0)$.
\item Consider the transformation $\vp_t : (x, s) \mapsto (x, s+t)$.
\item Then we can show that any $\vp_t$-invariant probability measure on $\overline Z$ is of the form $d\mu \times dt / \int l d\mu$ for a $T$-invariant measure $\mu$.
\end{enumerate}

\begin{prop}\label{prop:suspension}
Let $(Z,T)$ be a measurable transformation and assume that $\mathcal {M}(Z,T)$ is not empty. Let $l$ be a positive measurable function, bounded away from zero and infinity, from $Z$ to $\R$ and  consider the suspension  $(\overline {Z}, \vp) $ as above. For  $ {f}$ a nonnegative bounded function on $\overline {Z}$, define $f_Z$ a function on $Z$ by $$f_Z (x) \; = \; \int _0^{l(x)} f(x,s) ds.$$
 Then, 
 
 \begin{enumerate}
\item $\cP_{\overline Z} (f, \vp)$ is given by the unique number $t_0$ such that $\cP_Z (f_Z -t_0 l) = 0. $
 \item A probability measure $\mu$ on $Z$ realizes the maximum of $h(\mu) + \int (f_Z -t_0 l) d\mu$ if, and only if, the measure $dm= \frac{d\mu \times dt}{\int l d\mu}$ realizes the maximum of $h(m)+ \int f dm$.
 \end{enumerate}
 \end{prop}
 \begin{proof} By the definition of the pressure,
$$\mathcal{P}_{\overline {Z}}(f, \vp) = \underset{m}{\sup} \{ h_m + \int_{\overline {Z}} f dm \},$$ where the supremum is taken over all
Borel probability measures $m$ which are $\vp$-invariant on $\overline {Z}$. Any
$\vp$-invariant probability measure $m$ is of the form $$ m = \frac{\mu \times ds}{\int_{Z} l d\mu},$$ where $\mu$ is a $T$-invariant probability measure on $Z$, and $h_{m} = h_{\mu}/\int_{Z} l d\mu$ (\cite{Abr}).
It follows that
\begin{align*}
\mathcal{P}_{\overline {Z}}(f,\vp) &= \underset{\mu}{\sup} \{\frac{ h(\mu)}{\int_{Z} l d\mu} + \int_{Z} 
\int_{0}^{l}  f(x,s)\frac{d\mu}{\int_Z l d\mu} ds \} \\
&=\underset{\mu}{\sup} \{\frac{ h(\mu)}{\int_{Z} l d\mu} + \int_{Z} f_Z {\frac{d\mu}{\int l\; d\mu}} \}.
\end{align*}

Thus, $$\frac{h(\mu) + \int f_Z d\mu}{\int_Z l d\mu} \leq \cP _{\overline {Z}} (f, \vp) $$ i.e., $h_\mu+\int_Z  f_Z d\mu - \mathcal{P}_{\overline{Z}}(f, \varphi) \int_Z l d\mu \leq
0$ and the equality is attained by the supremum. Therefore $\mathcal{P}_Z (f_Z  - \cP_{\overline{Z}}(f,\vp) l)=0$. Conversely, suppose that a
positive constant $t$ satisfies $\mathcal{P}_Z(f_Z - t l)=0$, i.e., $\sup_{\mu} \{ h(\mu) + \int f_Z d\mu - t \int l d\mu\}=0.$
Then it follows that $$\underset{\mu}{\sup} \; \frac{h(\mu)+ \int f_Z d\mu}{\int l d\mu} =t ,$$ thus $t=\cP_{\overline Z}(f, \vp)$. 
\end{proof}



\section{Some characterizations of volume entropy}\label{s:characterization}



In this section, we show Theorem~\ref{th:pressure}. Recall that the volume entropy of the building is the topological entropy of the geodesic flow, which is the exponential growth rate of the number of longer and longer geodesic segments (that are separated enough) in the building. Theorem~\ref{th:pressure} tells us that we can count geodesic segments in any regular building by counting them in one apartment, each weighted by the number of preimages in the building under the retraction map.

As a consequence, we also show that
the entropy of a regular hyperbolic or Euclidian building is characterized
as the unique positive constant for which the pressure of some
function, now defined on some shift space, is zero, an analogous result to the tree case. 

Let $\Delta$ be a hyperbolic  or Euclidean regular building of type $(W,S)$ of dimension $n$, and let $P$ be a Coxeter polyhedron associated to $(W,S)$. 
We suppose that there exists a discrete group $\G$ acting totally discontinuously, isometrically, and cocompactly on $\Delta$.
Let $(X,g)$ be the compact quotient $\G \bs \Delta$ of $\Delta$, as in Section~\ref{sec:topentropy}.

Let $\cG(X) = \Gamma \bs \cG(\Delta)  $ be the space of geodesics on $X$ and $\vp _t $ the geodesic flow. By Theorems \ref{thm:Enrico} and \ref{thm:variational}, $h_{\Vol}$ is given by $\underset{ m \in \mathcal{M}(\vp_{t}) }{\sup} \{ h_m  \} ,$
where $\mathcal{M}(\vp_t)$ is the set of all $\vp$-invariant probability measures on $\cG(X)$. Consider the polyhedron $P$ as $W \bs \X$. The group $W$ acts  properly on $T^1 (\X)$ and, since $W$ acts by isometries on $\X$, the action on $T^1 (\X)$ commutes with the geodesic flow. Set $\cZ := W \bs T^1 (\X)$ and $\phi _t $ for the quotient action of the geodesic flow. We define the natural mapping $\pi $ from $\cG(X)$ to $\cZ$: for $v = \gamma (0) \in \cG(X)$, take an apartment $\A$ containing  a lift $\tilde v$ of $v $  to $\cG(\Delta) $  and take for $\pi (v) $ the class of $\tilde v$ in $W \bs T^1 (\X) = \cZ $. The element $\pi (v) $ does not depend of the choices of the lift of $v$ and of the apartment, and we have $ \pi (\vp _t  v) = \phi _t (\pi v).$ By Theorem \ref{thm:RVP}, we obtain:
\[ h_{\Vol }(X) \; = \;  \underset { \mu \in \mathcal{M}(\phi_{t}) } \sup \left\{h_{\mu} + \int _{\cZ} h_{\mathrm{top}} (\pi ^{-1} (z), \vp _t) d\mu (z) \right\}.\]

We have reduced the computation of $h_{\Vol} (X) $ to that of the pressure of the function $h_{\mathrm{top}} (\pi ^{-1} (z), \vp _t) $ on $\cZ.$

 In the next subsection, we will see that the entropy $h_{\mathrm{top}} (\pi ^{-1} (z), \vp _t) $ is roughly the exponential growth rate of the number of longer and longer geodesic segments (which are separated enough) in the building with initial vectors projecting to $z$.


\subsection{ Pressure of the fiber entropy}\label{sec:pressure}

In order to count geodesics segments in the building with the same projection, we define a bounded function $f$ on $T^1(\X)$ as follows. 

Consider the set $T^1(\mathcal{H})$ where $\mathcal{H}$ is the union of panels of the tessellation of $\X$ by images of $P$ under $W$. 
 
We first define $f$ on $T^1(\X)^o:=T^1(\X) - T^1(\mathcal{H})$. If $v \in T^1(\X)^o$, for $H \in \mathcal{H}$, define $t_H(v)$ to be the unique $t$ such that $\gamma _v (t) $ belongs to $H$, if it exists.
\begin{itemize}
\item Suppose $\gamma_v$ intersects only one $H$ at each time $t_H$. 
Define $$f(v)= \underset { H \in \mathcal{H}}{\sum} \ln q(H) \max \{ 0, 1- |t_H(v)| \},$$
where $q(H)+1$ is the thickness of $H$.
It is a finite sum as $\gamma_v$ intersects only a finite number of panels between time $-1$ and 1. 
\item Suppose $\gamma_v$ intersects several $H$s at the same time. By definition of buildings, the number $q(v, \varepsilon)$ of preimages of $\gamma_v([t_H(v)-\varepsilon, t_H(v)+\varepsilon])$ equals the number of $n$-dimensional $(2\varepsilon)$-discs in the building containg $\gamma_v([t_H(v)-\varepsilon, t_H(v)+\varepsilon])$. Thus the number $q(v, \varepsilon)$ is a continuous function of $v$. (Note that in the case of right-angled Fuchsian buildings, this number is simply $q(H_i)q(H_j)$ if $\gamma_v$ intersects $H_i$ and $H_j$.) Therefore, the function $f$ extends continuously to $T^1(\X)^o$.
\end{itemize}

Define $f$ analogously on $T^1(\mathcal{H})$, except that in the sum in the definition of $f$, we exclude all $H$'s such that $v$ belongs to $T^1(H)$. The resulting function $f$ is not continuous, precisely on $T^1(\mathcal{H})$, but it is a bounded measurable function on $T^1(\X)$. Moreover, it is a $W$-invariant function. Let us denote the function on $\mathcal{Z}$ again by $f$. The function $f$ approximates the ``weighted characteristic function'' of panels $\sum_{F} \ln q(F) \chi_F$ on the space of geodesics $\cZ$. 


\begin{prop}\label{prop:pressure} Let $X$ be a compact quotient of a regular Euclidean or hyperbolic building.
The volume entropy $h_{\Vol}(X)$ is equal to the pressure
$$ h_{\Vol}(X) = \mathcal{P}_{\cZ}(f, \phi_t),$$ of the geodesic flow $\phi_t$ on $\cZ$ with respect to the function $f$ defined above.
\end{prop}

Let us denote $t^+_v$ ($t^-_v$) the smallest $t \geq 0$ (largest $t < 0$, respectively) such that $\gamma_v(t)$ intersects the panels of the tessellation transversally.
\begin{defi}{\rm[definition of $l$ and $q$]}\label{def:l} 
For $v \in \cZ,$ 
\[ l(v) := d(\gamma_v(t^+_v), \gamma_v(t^-_v) ) = t^+_v - t^{-}_v. \]

If $v$ intersects only one panel transversally, then we define $q(v)+1$ to be the thickness of that face. If $v$ intersects several panels, $q(v)$ is $q(v, \varepsilon)$ for $t^{+}_v$ (defined in the beginning of Section~\ref{sec:pressure}). 
\end{defi}
We define $q(v)$ similarly when $v$ is contained in $T^1\mathcal{H}$:  again denote by $t^{+}_v$ the first time $t \geq 0$ such that $\gamma_v(t)$ intersects the panels of the tessellation transversally, and define $q(v)$ as above.

\begin{cor}\label{coro:pressure}
The function $\ln q/l$ is essentially cohomologous to $f$, thus
\[ h_{\Vol}(X) = \mathcal{P}_{\cZ}(\frac{\ln q }{l}, \phi_t). \]
\end{cor}

\begin{proof} \textit{(of Proposition~\ref{prop:pressure})} Let $f$ be as above. We first claim that, for all $v \in \cZ$,
$$ h_{\mathrm{top}} (\pi ^{-1} (v), \vp _t) \;= \; \underset{ T \to \infty}{\limsup} \frac{1}{T} S_Tf (v). $$

Fix $\e $ positive and small and $T$ large. Fix a retraction map $\rho : \Delta \to \A$ centered at a chamber $C$ containing a lift $\tilde v$ of $v$.  Let us denote the extension of the retraction map to the space of geodesics $\cG_C(\Delta) \to \cG(\A)$ again by $\rho$, where $\cG_C(\Delta)$ is the set of geodesics passing through the chamber $C$.
The geodesic $\widetilde{\g}_{v}$ lies in $\A$ and the set $\pi^{-1} (v)$ is the set of geodesics $\rho^{-1} (\widetilde{\g} _{v})$ in $\cG(\Delta)$. All the geodesics in $\rho ^{-1} (\widetilde{\g} _{v})$ coincide at time zero and intersect the panels of the tessellation at the same sequence of times. For any two such geodesics, there is a greatest interval of time including time zero on which they coincide. 
\begin{figure}[htb!]
\centering%
\includegraphics{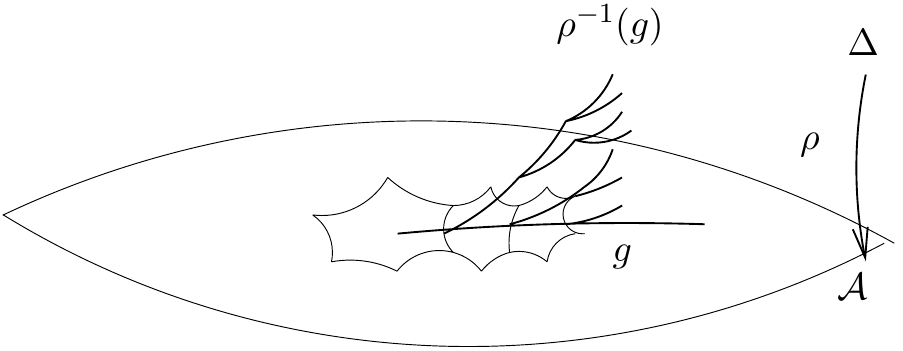}
\caption{Inverse image of a geodesic under the retraction}
\label{fig:building}
\end{figure}

Let us first assume that $\g_v$ is not contained in $T^1(\mathcal{H})$, and that it intersects only one panel at a time.
Then, two geodesics in $\rho ^{-1} (\widetilde{\g} _{v})$, at each common intersection with the panels of the tessellation, might bifurcate or not. (see Figure~\ref{fig:building}). 

 There is a function $\e \mapsto a(\e)$ such that, if $\g $ and $\g '$ coincide for $s, |s| \leq a(\e/2) $, then $d_\cG (\g, \g')< \e$. Conversely, if $d_\cG (\g, \g ' )< \e$, then $\g, \g'$ coincide for $s, |s| \leq a(\e)$. Explicitly, $a(\e) $ is such that $\int _{a(\e)}^\infty 2t e^{-t} dt = \e.$ It follows that  if a set of geodesics in $\rho ^{-1} (\widetilde{\g}_{v})$ is  $(T,\e)$-separated, then any two geodesics in that set have bifurcated at some time $s$ , $ -a(\e/2) \leq s \leq T +a(\e/2) $. In the same way, a $(T,\e)$-spanning set of geodesics in $\rho ^{-1} (\widetilde{\g}_{v})$ must take into account all bifurcations taking place at times $s$, $-a(\e) \leq s \leq T + a(\e).$ This holds also for maximal $(T,\e)$-separated sets and minimal $(T,\e)$-spanning sets, so that:
\begin{align*}
 \# [\rho ^{-1} (\widetilde{\g}_{v}([-a(\e), T+a(\e)]))] \leq \;  & S_{d_{\cG}} (\vp, \e , T, \rho ^{-1} (\widetilde{\g}_{v})) \leq \\
 \leq N_{d_{\cG}} (\vp, \e , T, \rho ^{-1} (\widetilde{\g}_{v}))& \leq \# [\rho ^{-1} (\widetilde{\g}_{v} ([-a(\e/2), T+a(\e/2)]))] . 
 \end{align*}


 By definition of a building, the cardinality of $\rho^{-1} (\g ((s,s')))$ is given by the product of  the thicknesses of all the panels crossed  by $\g(s,s')$. 
 
 Let   $\delta $  be the greatest nonpositive $t$ such that $\g _v(t) $ belongs to only one $(n-1)$-dimensional cells. Then, 
\begin{align*} \int_{\delta} ^{1+\delta} f(\phi_t(v))dt \; &= \; \frac{1}{2} \sum_{\g _v(\dd) \in F} \ln q(F) \quad \quad \quad \quad (*) \\
\textrm {and } \; \int_{l(v) +\dd -1} ^{l(v)+\delta} f(\phi_t(v))dt\; &= \; \frac{1}{2} \sum_{\g _v(l(v) +\dd )  \in F'} \ln q(F'). 
\end{align*}
Here, $F$ and $F'$ denote the last and the first panel that $\g _v$ meets before (or at) and after time $0$, respectively.
 
 By the above properties of the function $f$, it follows that for $a >0$:
 $$\int _{-a}^{T+a} f(\phi _t v) dt \; \leq \; \ln \left( \#[\rho ^{-1} (\g_{\tilde v } ([-a, T+a]))] \right) \; \leq \;  \int _{-a-1}^{T+a +1} f(\phi _t v) dt .$$
 
 By the definition of $f$ (in the beginning of Section~\ref{sec:pressure}), the same formula holds even if $\gamma_v$ intersects several panels at the same time, or if it is contained in $T^1(\mathcal{H})$ by the definition of the function $f$.

 Let $ \ln K$ be the maximum of the function $f$. It follows that 
 \begin{align*}
 S_T f(v)  + 2 a(\e) \ln K  &\leq S_{d_{\cG}} (\vp, \e , T, \rho ^{-1} (\widetilde{\g}_{v})) \leq \\
 &\leq N_{d_{\cG}} (\vp, \e , T, \rho ^{-1} (\widetilde{\g}_{v})) \leq   S_T f(v)  + (2 a(\e/2)  +2) \ln K.
 \end{align*}

 The claim  follows by taking $\limsup \frac{1}{T}$ of this inequality. Since the function  $f$ is bounded, we can apply the ergodic theorem with   any  $\phi _t$-invariant measure. This shows that the functions $f$ and $h_{\mathrm{top}} (\pi ^{-1} (v), \vp _t) $ are essentially cohomologous. By the discussion above,  the volume entropy $h_{\mathrm{vol}}(X)$ is equal to $\mathcal{P}_\cZ(f, \phi)$ and this proves Proposition \ref{prop:pressure}.
 \end{proof}
 
\begin{proof} \textit{(of Corollary~\ref{coro:pressure})}
We show that ${\ln q}/{l}$ on the space of geodesics is essentially cohomologous to $f$ by showing that, for any $v \in \cZ$, and any $T>0$,
$$ | S_T(\ln q /l)(v) - S_T (f) (v)|  \leq C \ln K, $$ for some constant $C$. 

For  $v \in \cZ$, let $\dd$ ($(T + \dd ')$) be the greatest nonpositive (smaller than $T$, respectively) time when $\g _v $ meets a panel of the tessellation. By definition, $\int _\dd^{T+\dd '}(\ln q /l)(v) $ is the logarithm of the product of the thicknesses of the panels met by the geodesic $\widetilde{\g}_{v} ([\dd, T+\dd ')). $ The difference with $S_T(\ln q /l)(v)$ is made of $\int _{\dd}^0 \ln q /l (\phi _t (v)) dt $ and  $\int _{T+ \dd'}^T \ln q /l (\phi _t (v)) dt $. Since by definition, $|\dd | \leq l (\phi _t v)$ for $\dd  < t \leq 0$, the first integral is bounded by $\ln K$. In the same way, the second integral is also bounded by $\ln K $.
Analogously, by the property $(*)$ of $f$, $\int _{\dd}^{T+\dd '}f (v) $ is the logarithm of the product of the thicknesses of the panels met by the geodesic $\widetilde{\g}_{v} ([\dd, T+\dd ')) $, up to an error $\frac{1}{2} \ln K$ at $\dd$ and $T+ \dd'$ (we assume that $T>1$).  Since $f$ is bounded by $\ln K$ and $|\dd| $,  $|\dd'|$ by the diameter of the chamber, the difference of $S_T f $ with $\int _{\dd}^{T+\dd '}f (v) $ is bounded by $C_1 \ln K$. 
Thus, setting $C = C_1 +2$, \[ | S_T(\ln q/l)(\g) -S_T(f)(\g) | \leq C \ln K. 
\]
Now by the claim, and the fact that $\int f d\mu = \int f \circ \phi_{t} d\mu = \frac{1}{T} \int_{0} ^{T} \int f \circ \phi_{t} dt d\mu=\frac{1}{T} \int S_T(f) d\mu$ for any $\phi_t$-invariant measure $\mu$,  we conclude that for any $T>0$,
\[ \left| \int \frac{\ln q}{l} d\mu -\int f d\mu \right| < C/T.  \]
 Therefore ${\ln q}/l$ is essentially cohomologous to $f$ and their pressures are equal.
\end{proof}

Theorem \ref{th:pressure} directly  follows from Corollary \ref{coro:pressure}. Since $\cZ = W\bs T^1(\X)$, $T^1 (Y) = W' \bs T^1 (\X)$ and $W'$ is a finite index subgroup of $
W$, the space $T^1 (Y)$ is a finite extension of $\cZ$ and the geodesic flow $\vp _t$ on $T^1Y$ projects to the flow $\phi _t $ on $\cZ$.  Each invariant measure on $T^1(Y)$ projects on some invariant measure on $\cZ$. Since the extension is finite, the entropy is preserved (apply e.g. Theorem \ref{thm:RVP}). The functions $f$ and $\ln q /l$ lift to functions which we denote the same way. In the proof of Corollary \ref{coro:pressure}, we showed that both $S_T f (v) $ and $S_T (\ln q/l) (v) $ are the number of preimages (under $\rho$) of geodesic segment $\g _v ([0,T))$, up to some bounded error. Therefore the limits $ \underset{T \to \infty}{\lim }\frac{1}{T} S_T f(v)$ and $\underset {T\to \infty}{\lim } \frac{1}{T} S_T ( \ln q/l) $ are the same when computed in $T^1(Y) $ or in $\cZ$. Given that the entropies are the same and the integral against invariant measures are the same, it follows that the pressure of the functions $f$ and $\ln q /l $ is the same on $T^1(Y)$ as on $\cZ$. By Proposition \ref{prop:pressure} and Corollary \ref{coro:pressure}, these pressures  coincide with the volume entropy of $X$.

\subsection{Coding of geodesics and the pressure of a subshift}\label{sub:coding}
Now let us give another characterization of the volume entropy of regular hyperbolic buildings in
terms of the pressure of a function on a subshift. The shift space $\Sigma$ we consider here is the set of geodesic cutting sequences on $\A$, which is not necessarily a subshift of finite type. (All we need here is a section of the geodesic flow.) Let us first describe it for regular hyperbolic buildings. 

Take as alphabet the set  of panels of $P$, and for all integer $k$, let $\Sigma _k$ be the set of cylinders based on the words $x_{-k}, x_{-k+1}, \cdots, x_0, \cdots, x_k$, of length $2k +1$, such that there exists at least  one geodesic $\g$ which intersects transversally the faces  $x_{-k}, x_{-k+1}, \cdots, x_0, \cdots, x_k$, in that order, and of which intersection with $x_0$ occurs at time 0. The intersection $\Sigma = \cap_k \Sigma _k$ is a closed subset of the set of biinfinite sequences of symbols. Let $\sigma $ be the shift transformation on the space of sequences: $\sigma : \ux \mapsto \sigma (\ux) =  \underline y, $ where $y_n = x _{n+1}$. The space $\Sigma $ is shift invariant. For each $\ux \in \Sigma $, there exists  a unique geodesic $p(\ux)$ which intersects the faces $x_{-k}, \cdots, x_k$ corresponding to any cylinder in $\Sigma _k $ containing $\ux$. It exists because a decreasing intersection of nonempty compact sets is not empty. It is unique because two distinct geodesics in $\Hyp$ cannot remain at a bounded distance from one another. Set $p: \Sigma \to \cZ$ for the mapping just defined. Endow the set of bi-infinite sequences of indices of faces with the product topology. The set $\Sigma $ is the intersection of the cylinders which contain it, therefore $\Sigma $ is a closed invariant subset. 

Now to have the geodesic flow as the suspension flow of $(\Sigma, \sigma)$, let us define the ceiling function of a bi-infinite sequence corresponding to a geodesic $\gamma$.
\begin{defi}{\rm [definition of ceiling function $L$ and the function $Q$]}\label{def:L} We define $L(\ux)$ to be the length of geodesic segment between the panels $x_0$ and $x_1$. Define $Q(\ux)$ so that $Q(\ux)+1$ is the thickness at $x_0$.
\end{defi}
 The functions $L$ and $Q$ are similar to the functions $l$ and $q$ defined in Definition~\ref{def:l}, except 
at intersection of several panels:
 if the base point of $p(\ux)$ belongs to two faces $x_0$, $x_1$, then 
 \[ L(\ux)=0, \quad Q(\ux) + Q(\sigma \ux) = q(x_0)+q(x_1) = q(p(\ux)),
 \]
 with  the definition of $q$ (in Definition~\ref{def:l}). Similar relation holds when the base point of $p(\ux)$ belongs to more than two panels.
  Let us denote the suspension space by $\Sigma_L$ and the suspension flow by $\psi$. 
 
 Although $L$ is not bounded from below by a positive number, there is a finite $N$ such that $\sum _0^N L (T^k x) $ is bounded from below by a positive number (on $Y$, there is a $N$ such that a geodesic intersecting $N$ panels is at least as long as the injectivity radius of $Y$) and Proposition \ref{prop:suspension} remains valid. 

For Euclidean buildings, there might be several geodesics corresponding to a sequence. However, given $\ux \in \Sigma$, $p(\ux)$ is a "tree-band", i.e. a compact convex set in $\R^n$ times the inverse image of one geodesic under the retraction map. Thus the entropy contributed by $p(\ux)$ is same as the entropy contributed by a geodesic.
 
 By Proposition \ref{prop:suspension}, we have:
\begin{cor}\label{cor:subshift} 
Let $X$ a compact quotient of a regular building.
the volume entropy $h_{\Vol}(X)$ is the unique positive constant $h$ such that
$$ \mathcal{P}_{\Sigma}(\ln Q-hL)=0,$$
where $\mathcal{P}_{\Sigma}(\ln Q-hL)$ is the pressure of the function $\ln Q -hL$, now on the space $(\Sigma, \sigma)$ of the
subshift.
\end{cor}
\noindent{\it Remark.} The characterization above is analogous to the characterization of the volume entropy on a finite
graph (\cite{Lim}), as $h_{\Vol}$ is the unique positive constant such that the system of equations $x=Ax$ (here $A$ is the edge adjacency matrix
multiplied by $e^{-hL(f)}$ to each $(e,f)$-term) has a positive solution. It is equivalent to saying that $h$ is the unique
positive constant such that the pressure $\mathcal{P}(-hL)$ of the function $-hL$ (on the space of subshifts of $A$) is equal to zero.




\section{Liouville measure and a lower bound}\label{s:lowerbound}


In this section we show Proposition \ref{prop:santalo} and Theorem~\ref{th:Liouville}. As before,  let $\Delta$ be a regular building defined in Section~\ref{ss:building}, and let $X$ be a compact quotient of $\Delta$. We continue denoting $\cZ$ for the quotient $W \bs T^1{\X}$ and $Y$ for $W' \bs  \X$, so that $T^1 (Y) $ is a finite cover of $\cZ .$

\subsection{Lower bound of volume entropy}\label{sec:lowerbound}
As recalled in the introduction, the lower bound of the volume entropy (Corollary \ref{cor:lowerbound}) follows from an integral computation (Proposition \ref{prop:santalo}).
This computation uses results  known as Santal\'o's formulas in integral geometry.


Let $P$ be a convex polyhedron in $\X$, either hyperbolic or Euclidean. Let us consider $P$ as a manifold with boundary. Every element $\g$ of the unit tangent bundle $T^1P$  can be specified by $\tilde{x} =(q,v,t)$, where $q \in \partial P$, $v \in S^{(n-1,+)}$ is a unit vector of direction (here, $S^{(n-1, +)}$ is the set of inward directions, which can be thought as the northern half of the unit $(n-1)$-- sphere), and $t$ is a real number between $0$ and $l(\g)$. In this way $(q, v,t)$ form a coordinate system of $T^{1}P$.
 Denote the Liouville measure  on $T^1P$ by $m_L$. 
 
\begin{pro} [Proposition \ref{prop:santalo}] 
$$\int_{T^{1}(P)} \frac{\ln q}{l} dm_L = c_{n} \underset{F}{\sum} \ln q(F) \Vol(F),$$
where $c_n$ is the volume of the unit ball in $\mathbb{R}^n$, and
where the sum is over the set of panels of $P$.
\end{pro}
\begin{proof} 
For two unit vectors $v,w$, let us denote by $v \cdot w$ the inner product (in $\R^n$).
Santal\'o's formula says that
the Liouville measure is
\begin{equation}\label{eqn:santalo} dm_L = (v \cdot \textbf{n}(q)) dq dv dt,
 \end{equation}
where $\textbf{n}(q)$ is a unit vector normal to the face of $P$ which contains $q$,  and $dq$ is the Lebesgue measure on $\partial P$ (see \cite{San1}, or \cite{San2} Section 19.1). 

By definition, the value of $\ln q(\g) $ for a given geodesic $\g =(q, v,t)$ depends only on the face $F$ which $q$ belongs to. Thus
\begin{align*}
\int_{T^{1}(P)} \frac{\ln q}{l} dm_L 
&= \int_{T^{1}(P)} \frac{\ln q}{l} (v\cdot \textbf{n}(q)) dt dv dq = \int (\ln q)  (v\cdot \textbf{n}(q)) dv dq \\
&= \sum_{F \; \mathrm{face \; of\;} P} \ln q(F) \Vol(F) \int _{S^{(n-1,+)}} (v\cdot \textbf{n}(q)) dv\\
&= c_n \sum_{F} \ln q(F) \Vol (F),\\
\end{align*}
where $c_n =  \int _{S^{(n-1,+)}}(v\cdot \textbf{n}) dv$ (where $\mathbf{n}$ is the unit vector directing the north pole of $S^{n-1}$).
Using standard hyperspherical coordinate, it is easy to see that $c_n = \frac{1}{2\pi} \mathrm{Vol}(S_{n+1}) = \mathrm{Vol}(B_n)$.
 \end{proof}

Combining Theorem \ref{th:pressure} and Proposition \ref{prop:santalo}, Corollary \ref{cor:lowerbound} follows.



\subsection{Liouville measure and Bowen-Margulis measure}\label{sec:baratin}
\vspace{.1 in} \noindent 
In the previous subsection, we computed the maximum of entropy of measures on $\cG (X)$ which come from the Liouville measure on apartments.
Recall that  $(\cG (X), \vp _t) $ is the geodesic flow of the CAT(-1) space  $X$. By \cite {Ro}, there is a unique measure of maximal entropy, which is ergodic and mixing.
We call it the Bowen-Margulis measure.  

In this subsection we show the first part of Theorem~\ref{th:Liouville}. From now on, let $\Delta$ be a regular hyperbolic building defined in Section~\ref{ss:building}, and let $X$ be a compact quotient of $\Delta$.

\begin{theo}\label{thm:liouville} The Bowen-Margulis  measure on any compact quotient of a regular hyperbolic building does not project to the Liouville measure.
\end{theo}

In projection on $\cZ$, the statement of Theorem \ref{th:Liouville} is that the projection of the Bowen-Margulis measure from $\cG (X)$ and the projection of the Liouville measure from $T^1 (Y) $ are not the same. Both are ergodic measures, and Liouville measure has a unique lift to $T^1 (Y)$. If we assume those measures are the same, then the lift of the Bowen-Margulis measure to $T^1 (Y)$ is the Liouville measure. In other words, in that case, by Theorem \ref {th:pressure}, the Liouville measure on $T^1 (Y) $ realizes the maximum of  $(h_m + \int (\ln q/l) dm)$ over all $\phi _t $-invariant measures. 

We are going to show that this leads to a contradiction through  a cohomological argument. In the previous section, we represented  the geodesic flow on $T^1 (Y) $ as the suspension flow of a subshift  $( \Sigma, \sigma )$ by using the geodesic cutting sequence.
 A consequence of Corollary \ref{cor:subshift} is that the $\sigma$-invariant measure associated to the Bowen-Margulis measure is the  equilibrium measure for the function $\ln Q -h_{\Vol }(X)L$. On the other hand, by Proposition \ref{prop:suspension}, since the Liouville measure is the measure of maximal entropy, which is $(n-1)$, the measure associated to the Liouville measure is the equilibrium measure for the function $-(n-1) L$. 
For subshifts of finite type, R. Bowen showed  that  two H\"older continuous functions $f$ and $g$ have   the same equilibrium measure (i.e. the pressure $\mathcal{P}(f)$ of $f$ is attained by the same $\sigma$-invariant measure as the pressure $\mathcal{P}(g)$) if and only if $f$ and $g$ are cohomologous up to a constant. In particular, if $f$ and $g$ have the same equilibrium measures, then for any periodic orbit  $\g$,
\[ \int_\g f dm = \int_\g g dm + c, \]
where $dm$ is the normalized counting measure on $\g$ and $c$ is a constant independent on $\g$, namely $c = P_{\Sigma }(f, \sigma) - P_{\Sigma }(g, \sigma)$ 
 (see, for instance, \cite{PP} Proposition 3.6). 
 
\vspace{.1 in}
\noindent \textit{Remark.} It was pointed out by the referee that the function $f$ we defined in Section \ref{sec:pressure} cannot be a H\"older-continuous function on $\cG(X)$. Here is the argument of the referee.
Indeed if $f$ is continuous, then
the cross-ratio associated to log q / l is continuous on the boundary of
the hyperbolic space, but using the Coxeter walls structure one may  prove
that the cross-ratio values belong to a countable set. We will show that $\ln Q$ and $L$ are H\"older continuous on $\Sigma$.

 \vspace{0.1 in}
 
 Here, the subshift $\Sigma $ is not of finite type, but for hyperbolic surfaces, there is a subshift of finite type \`a la Bowen-Series \cite{BS} such that the subshift $\Sigma$ is a finite-to-one factor of  it. This case of surfaces will be treated in Section \ref{sec:surface}. 
By Bowen's argument, if the Bowen-Margulis measure projects to a Liouville measure, then for any periodic orbit $\g= (x_1, x_2, \cdots, x_k) ^\infty $ in $\Sigma $,
\[ \sum _{i=1 }^k \ln q_{x_i} =  ( h_{\Vol} - n +1 ) \sum _{i=1}^k l(\sigma ^i \g). 
\] 

We arrive at a contradiction by examining this relation for periodic orbits of $\vp_t$ corresponding to special closed geodesics.
In the general case, we do not have a subshift of finite type, but comparing directly the conditional measures of Bowen-Margulis measure and those of Liouville measure, we arrive at a cohomological equation for $\log Q - (h_{\Vol}-n+1)L$, true only almost everywhere. Using the cross ratio on the boundary, this suffices to arrive at a contradiction. 



\subsection{Proof of Theorem~\ref{thm:liouville} for surfaces}\label{sec:surface}

Let $(\Sigma, \sigma)$ the subshift we defined in Section~\ref{sub:coding}. For surfaces, by Series \cite{S1}, it is conjugate to a sofic system, i.e. a finite-to-one factor of a subshift $(\Sigma', \sigma')$ of finite type. The functions $Q$ and $L$ defined in Definition~\ref{def:L} are H\"older-continuous on $\Sigma$. Indeed,  $Q$ depends only on the zero-th coordinate. Also, if two geodesics $\gamma_1, \gamma_2$ are coded by sequences $\ux, \uy$ such that $\ux_i = \uy_i$ for $|i| < k$, then there are some constants $c_1, c_2$ such that $d(\gamma_1(t), \gamma_2(t))$ is smaller than the diameter of $P$ for $|t| < c_1k-c_2$. By hyperbolicity, It follows that
$d(\gamma_1(0), \gamma_2(0) < C e^{-c_1k}$. Therefore $$|L(\gamma_1)-L(\gamma_2)| < C'e^{-c_1k}.$$

We denote the functions on $\Sigma'$ associated to $Q$ and $L$ by $Q'$ and $L'$. Since the factor map and the conjugacy preserve H\"older-continuous maps, these functions $L', Q'$ are H\"older-continuous on $\Sigma'$ as well.

On $\Sigma_L$, the Liouville measure $m_L$ (Bowen-Margulis measure $m_{BM}$) is of the form $ m_L = \frac{\mu _0 \times dt } {\int l d\mu _0} $( $ m_{BM} = \frac{\mu _{BM} \times dt } {\int l d\mu _{BM}} $, respectively) for some shift-invariant probability measure $\mu _0 $ ($\mu_{BM}$, respectively)  on $\Sigma$. Let us denote the measures on $\Sigma'$ corresponding to $\mu_L, \mu_{BM}$ by $\mu'_L, \mu'_{BM}$, respectively.
 
As explained in the beginning of Section~\ref{sec:baratin}, the measure $\mu'_{BM}$ (the measure $\mu'_0$) is the equilibrium measure for the function $\ln Q' - h_{\textrm{vol}}L'$ (the function $-L'$, respectively). If the functions $\ln Q' - h_{\textrm{vol}}L'$ and $-L'$ have the same equilibrium measures and the same pressure, then they are cohomologous. It follows that $\ln Q'$ is cohomologous to $(h_{\mathrm{vol}}-1)L'$. In particular, at a periodic orbit $\g = (x_1, x_2, \cdots, x_k )^\infty$, we get

\[ \sum _{i=1 }^k \ln q_{x_i} =  ( h_{\Vol} - 1 ) \sum _{i=1}^k l(\sigma ^i \g). \] 

Going back to closed geodesics,
\[ \ln M(\g) \; = \; ( h_{\Vol} - n +1 ) l (\g), \]

where $M$ is the multiplicity of a closed geodesic, i.e. the number of preimages under $\rho$ of a period.
We derive a contradiction by constructing a family of closed geodesics on which $\ln M$ is linear but $l$ is not.

Consider two elements $A, B$ of the fundamental group $W'$ of $W' \backslash \mathbb{H}^2$, and let $g_k$ be the unique closed geodesic in the free homotopy class of $[A^kB]$.
\begin{figure}[htb!]
\centering
\includegraphics{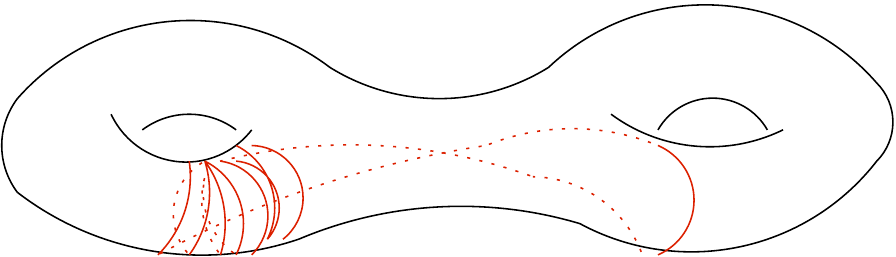}
\caption{A geodesic representing $[A^kB]$}
\label{fig:surface}
\end{figure}
After change of basis, if necessary, we may assume that 
$A= \left( \begin{array}{cc}
\lambda & 0 \\
0 & \lambda^{-1}  \end{array} \right),$
$ B=\left( \begin{array}{cc}
a & b \\
c &d  \end{array} \right),$ where $ (a^2+b^2) (c^2 +d^2) \neq 1.$ (Here $a,b,c,d, \lambda \in \mathbb{R}$ depend on both the number of faces of $P$ and the metric on $P$). 

The length of $g_k$ is given by
\begin{align*} l(g_k) &= \cosh^{-1} \left( \frac {\{\tr (A^k B)(A^kB)^t\}}{2} \right)\\
& = \cosh^{-1} \left( \frac{\lambda ^{2k} (a^2 +  c^2) +\lambda ^{-2k}( b^2 + d^2)}{2} \right),\end{align*}
which is not a linear function of $k$.

\begin{figure}[htb!]
\centering
\includegraphics{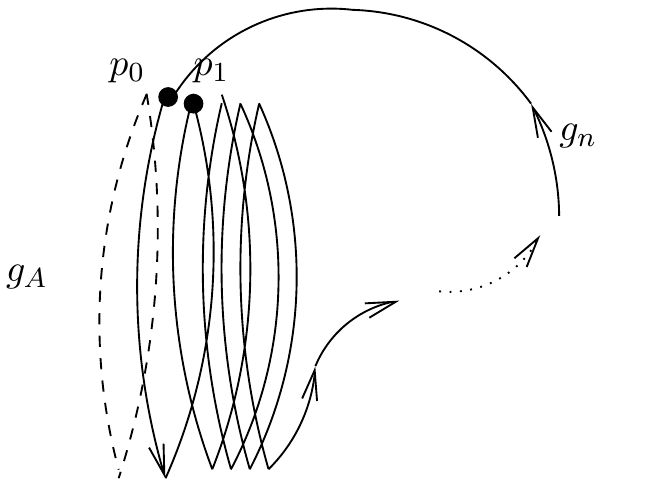}
\caption{Closing a geodesic}
\label{fig:spiral}
\end{figure}

On the other hand, as $g_k$ is a simple closed curve, for sufficiently large $k$, the closed geodesic $g_k$ has to spiral around the closed geodesic $g_A$ representing $[A]$. Let $p_0$ be the point of $g_k$ closest to $g_A$. Remove an interval of radius $l(g_A)/4$ around $p_0$ and let $p_1$ be the point closest to $p_0$.
Now let us reparametrize the geodesic $g_k$ so that $g(0)=p_1$, $g(t)=p_0$, and $g(T)=p_1$.
Remove the last part $p_0 p_1$ of the segment so that the remaining geodesic segment $g_k [0, t]$ have endpoints $p_0$ and $p_1$ which are $\e$-close.

 By Anosov closing lemma, there exists a closed geodesic $g'$ which is $\e$-close to $g_k[0,t]$ ( i.e., $|g'(s)-g_k(s)| < \e$, $\forall 0 \leq s \leq t$), and which is in the homotopy class of $A^{k-1}B$. By the uniqueness of such closed geodesic, $g'=g_{k-1}$. In other words, for sufficiently large $k$, $g_k$ follows the trajectory of $g_{k-1}$ and then spiral around the geodesic representing $[A]$ one more time. Therefore we can assume that for every $n$ sufficiently large, $\ln M({g_k} ) = C_1 k + C_2$, for some constant $C_1$ and $C_2$ ($C_1, C_2$ are functions of the thickness $q$'s and the metric polygon $P$). 
Thus $\ln M$ is a linear function of $k$, whereas the function $l$ is not.




\subsection{The proof in the general case}\label{sec:thermo}

Let $(\Sigma, \sigma)$ the subshift we defined in Section~\ref{sub:coding}. The geodesic flow on $\cZ$ is represented as a suspension flow above the section $p(\Sigma) $, with $L$ as the ceiling function. In this representation, the Liouville measure $m_L $ is of the form $ m_L = \frac{\mu _0 \times dt } {\int L d\mu _0} $ for some smooth measure $\mu _0 $ on $p(\Sigma )$.
  
\begin{lem} For $\mu _0$-almost every vector $v \in p(\Sigma )$, there is a unique $\underline {x}$ in $\Sigma$ such that $p(\ux ) = v$. For $v \in p(\Sigma ) $,   let $W^s_{loc} (v)$ be the set of vectors  $v' \in p(\Sigma ) $  such that there are  $\ux $ and $\ux '$ in $\Sigma $ with  the same nonnegative entries and  $p(\ux ) = v, p(\ux ' ) = v'.$ For Liouville-almost every sequence $v$, the set  $W^s_{loc} (v)$ is a  connected   $(n-1)$-dimensional ``suborbifold''  of $Z$, transverse to the weak unstable manifold of $v$. More precisely, on the finite cover $\overline{\pi} : T^1(Y) \to \cZ$, the inverse image $\overline{\pi}^{-1}(W^s_{loc} (v))$ is a connected $(n-1)$-dimensional submanifold.
\end{lem}
\begin{proof} Let us prove the statement on $T^1(Y)$.
There is a finite family  $\LL$ of  totally geodesic closed submanifolds of $Y$ which are the images of the panels of the tessellation. By construction, the vectors in $p(\Sigma ) $ are based on $\LL$. Since it forms a lower dimensional submanifold, the set $T^1(\LL) $ is negligible for the smooth measure $\mu _0$. Let $v$ be a vector based on $\LL$ but not belonging to $T^1 (\LL)$, and consider $W^0(v)$ the set of vectors based on the same face in  $\LL$  and such that the geodesics $\g _{\widetilde{v}} $ and $\g _{\widetilde{v}'}$ satisfy $\sup _{t \geq 0 } d(\g _{\widetilde{v}}  (t) , \g_{\widetilde{v}'}(t) )< \infty $. $W^0 (v) $ is a  connected   $(n-1)$-dimensional submanifold  of $T^1 (Y)$, transverse to the weak unstable manifold of $v$. The set  $W^s_{loc} (v)$ of vectors such that the associated sequence has the same nonnegative entries as the sequence of $v$ is a convex (hence connected) subset of $W^0(v)$. We claim that,  for Liouville almost every $v$, $W^s_{loc }(v)$ is a neighborhood of $v$ in  $W^0(v)$. In particular, it is an $(n-1)$-dimensional manifold. 

The argument is classical for dispersing billiards (\cite {Sin}): let $t_1, t_2 ,\cdots, t_n, \cdots $ the instants when $\g _v (t_n ) $ intersects $\LL$. There are constants $c_1,c_2$ such that $t_n \geq c_1 n - c_2$. Let $\dd_n (v) $ be the distance from $\g_v(t_n) $ to the set of unit vectors in $T^1 (Y)$ based on the $(n-2)$ dimensional boundaries of the faces.  If, for all positive $n$,  $v' \in W^0(v) $ satisfies $d_{T^1(Y)} (\g_{v'}(t_n+s), \g_v(t_n)) < \dd _n /3$, for some $s, |s| \leq \dd /3$, then $v' \in W^s_{loc} (v). $  Since $d_{T^1(Y)} (\g_{v'}(t_n), \g_v(t_n))  = e^{-t_n} d_{T^1(Y)} (v,v')$, it follows that $W^s_{loc}$ contains a neighborhood of $v$ in $W^0 (v) $ as soon as $\inf _n \dd_n e^{c_1n} >0$. The Liouville measure of a $\dd $ neighborhood of a codimension 1 subset  is $O(\dd ).$ By invariance of the geodesic flow and a Borel-Cantelli argument, Liouville almost every $v \in T^1(Y)$ satisfies  $\inf _n \dd_n e^{c_1n} >0.$ This shows the second part of the lemma. The first part follows by considering the successive preimages of $W^s_{loc} (v(\underline {x})).$
\end{proof}
The geodesic flow $\phi _t$ on $Z$ is coded by the suspension flow on $(\Sigma, \sigma)$, with ceiling function $L$.

We will prove in Section~\ref{sec:bressaud}:
\begin{prop}\label{prop:general}
Assume that the Bowen-Margulis measure projects on the Liouville measure. 
Then there is a function $u$ which is H\"older-continuous on each $W^s_{loc}(\ux)$ and which satisfies
$$\log Q(\ux) -(h_{\Vol}-(n-1))L(\ux) = u(\ux) - u(\sigma \ux), a.e.$$
Similary, there is a function $u'$ which is H\"older-continuous on each $W^u_{loc}(\ux)$ and which satisfies
$$\log Q(\ux) -(h_{\Vol}-(n-1))L(\ux) = u'(\ux) - u'(\sigma \ux), a.e.$$
These functions $u$ and $u'$ coincide almost everywhere.
\end{prop}
Let us assume Proposition~\ref{prop:general} for now and prove Theorem~\ref{thm:liouville}.
\begin{proof} [Proof of Theorem~\ref{thm:liouville}]
Assume that the Bowen-Margulis measure projects on the Liouville measure. We will get a contradiction by comparing the cross ratio of functions $L$ and $\frac{\log Q}{h_{\Vol} -n+1}$ which are cohomologous almost everywhere.
Let us express the cross-ratio on the set of quadruples of points which are pairwise distinct $\partial(\A)^{4}_0$ in two ways: first the usual way using $L$, and the other way using the function $\log Q/(h_{\Vol} -n+1)$.

Let $\zeta_1, \zeta_2, \zeta_3, \zeta_4$ be Lebesgue almost every points in $\partial \A$.  
Let us denote the geodesic in $\A$ with extreme points $\zeta_i, \zeta_j$ by $\g_{i, j}$. Let us choose sequences of unit vectors $x^k_i, y^k_i$ for $k=0,1,2,\cdots$ and $i=1,2,3,4$, whose base points belong to the panels of the tessellation, and such that 
\begin{align*}
x_i, x_j \in \g_{i,j}, \;\;& \textrm{where} \;\; (i,j)=(1,4),(2,3),\\
 y_i, y_j \in \g_{i,j}, \;\;& \textrm{where} \;\; (i,j)=(1,3), (2,4),\\
\textrm{and} \;\;x^k_i \to \zeta_i,  y^k_i \to \zeta_i, &\;\;\textrm{as}\; k \to \infty.
\end{align*}
Choose $K$ large enough so that
 $x^k_i, y^k_i$ are in the same $W^s_{loc}$ for $k \geq K$ and $i=1,2$, and $x^k_i, y^k_i$ are in the same $W^u_{loc}$ for $k \geq K$ and $i=3,4$.
Define the cross-ratio
\[ [\zeta_1,\zeta_2,\zeta_3,\zeta_4] = \underset{k \to \infty}{\lim} \left\{d(y^k_1, y^k_3)+d(y^k_2,y^k_4) -d(x^k_2, x^k_3)-d(x^k_1,x^k_4) \right\},
\]
where $d(x,x')$ denotes the distance between the base points of $x,x'$.
Along $\g_{1,3}$, the distance $d(y^k_1, y^k_3)$ is the sum of length $L(\sigma^i \uz)$ where $\uz$ is the cutting sequence of $y^k_3$ up to the base point of $y^k_1$. Therefore,
\[ d(y^k_1, y^k_3) = \sum L(\sigma^i \uz) = \sum \frac{\log Q(\sigma^i \uz)}{h_{\Vol} - n + 1} +u(y^k_3)-u(y^k_1)
\]
We have analogous equations for other pairs of vectors.
As $k \to \infty$, in the formula of the cross-ratio, there are only finitely many $\log Q$-terms which do not cancel. Moreover, by Proposition~\ref{prop:general}, since $x^k_i, y^k_i$ belong to the same $W^s_{loc}$ for $i=1,2$ ($W^u_{loc}$ for $i=3,4$, respectively) as $k \to \infty$,
\[
u(x^k_i)-u(y^k_i) \to 0, \forall i.
\]
It follows that the cross-ratio $[\zeta_1,\zeta_2,\zeta_3,\zeta_4]$ takes only countably many values on a set of full measure, a contradiction.
\end{proof}




\subsection{Proof of Proposition~\ref{prop:general}} \label{sec:bressaud}

\begin{pro} \textrm{[Proposition \ref{prop:general}]}
Assume that the Bowen-Margulis measure projects on the Liouville measure. 
Then there is a function $u$ which is H\"older-continuous on each $W^s_{loc}(\ux)$ and which satisfies
$$\log Q(\ux) -(h_{\Vol}-(n-1))L(\ux) = u(\ux) - u(\sigma \ux), a.e.$$
Similary, there is a function $u'$ which is H\"older-continuous on each $W^u_{loc}(\ux)$ and which satisfies
$$\log Q(\ux) -(h_{\Vol}-(n-1))L(\ux) = u'(\ux) - u'(\sigma \ux), a.e.$$
These functions $u$ and $u'$ coincide almost everywhere.
\end{pro}

\begin{proof}
We first construct the conformal measure on $\Sigma^+$ from Patterson-Sullivan measure on the boundary of the building $\Delta$. 

Since $\Delta$ is a CAT(-1)-space, there is a construction of Patterson-Sullivan measure \cite{CP}, which is a family of measures $m_x$, $x \in \Delta$ such that
\[
\frac{dm_{x'}}{dm_x} (\xi) = e^{-h_{\Vol} \beta_{\xi}(x',x)},
\]
where $\xi \in \partial \Delta$ and $\beta_{\xi}$ is the Busemann function based on $\xi$. Pick the origin $o \in C$ and the Patterson-Sullivan measure $m_o$.
For given $\ux=(x_0, \cdots, x_k, \cdots)$, recall that $W^s_{loc}(\ux)$ is the set of geodesics in $\cG(\A)$ whose geodesic cutting sequence is $(x_0, \cdots, x_k, \cdots)$ at time $0, T_1, \cdots, T_k, \cdots$.
Let us define $m_{\ux}$ on $W^s_{loc}(\ux)$ : for given $B \subset W^s_{loc}(\ux)$, take the Patterson-Sullivan measure $m_o$ of the set $\partial B$ of endpoints (at time $-\infty$) of the geodesics in $\cG(\Delta)$ which projects on $B$. 

Consider the map on the set of geodesics which is the return map of the geodesic flow composed with the reflection map, which corresponds to the shift map $\s$. More precisely, let us define the map $\widehat{\sigma} : W^s_{loc}(\ux) \to W^s_{loc}(\sigma\ux)$ by 
$$v \mapsto s_{x_1} \phi_{L(p^{-1}(v))} (v),$$
where $s_{x_1}$ is the reflection with respect to the face $x_1$.
For $y \in W^s_{loc}(\ux)$, let us denote by $\zeta(y)$ the endpoint at $-\infty$ of $y$.
Therefore, for $B \subset W^s_{loc} (\sigma \ux)$, 
\[
 m_{\sigma \ux}(B) = q(x_1) \int_{B} \frac{dm_{s_{x_1}o}}{dm_o} (\zeta(y)) d\widehat{\sigma}_{*}m_{\ux}(y),
\]
since $m_{\sigma \ux}(B)$ is the Patterson-Sullivan measure $m_{s_{x_1}o}$ (on $\partial \Delta$) of $q(x_1)$ copies (branched at $x_1$) of sets of endpoints of geodesics projecting to $B$. By the property of Patterson-Sullivan measure, 
\[
\frac{dm_{s_{x_1}o}}{dm_o} (\zeta) = e^{-h_{\Vol}\beta_{\zeta}(o, s^{-1}_{x_1}o)},
\] 
for $\zeta$ the endpoint at $-\infty$ of a geodesic projecting to $B$. For all geodesics projecting to the same geodesic $y \in B$, their endpoints at $-\infty$ project to the same point $\zeta(y) \in \partial \A$.
In other words, on $W^s_{loc}(\sigma x)$,
\[ 
\frac{dm_{\sigma \ux}}{d \widehat{\sigma}_{*} m_x} (y) = q(x_1) e^{-h_{\Vol}\beta_{\zeta(y)}(o, s^{-1}_{x_1}o)}.
\]

The following lemma is classical :

\begin{lem}\label{lem:bowen} There is a H\"older continuous function $L'$ on $\Sigma$ which is essentially cohomologous to $L$ with a transfer function $v$. On each $p^{-1}(W^{s}_{loc}(\ux))$, the function $v$ is H\"older-continuous.

\end{lem}

\begin{proof} Recall that $C$ is a fundamental domain for the action of $W$ on $\mathbb {H} ^n$, and $o \in C$. For $ {\ux} \in \Sigma $, denote $b_0( {\ux}) $ the footpoint  of $p(\ux)$, $b_1(\ux)$ the footpoint of $\g _{p(\ux)} (L(\ux))$. There is an element $s \in W$ such that $s(b_1(\ux) ) = b_0 (\sigma \ux)$. Denote $\zeta (\ux) = \g _{p(\ux)} (-\infty)$ the point at $- \infty$ in $\partial \mathbb {H}^n $ for $\g _{p(\ux)}$, and $B_{\zeta(\ux)} $ the Busemann function. Then we have:
\begin{align*}
L(\ux) = B_{\zeta (\ux)} (b_0(\ux), b_1(\ux)) & = B_{\zeta(\ux) } (b_0(\ux) , o )  + B_{\zeta (\ux) } (o, s^{-1} o) + B_{\zeta (\ux) } (s^{-1} o, b_1(\ux))\\
& = B_{\zeta(\ux) } (b_0(\ux) , o )  + B_{\zeta (\ux) } (o, s^{-1} o) - B_{s \zeta (\ux) } (s b_1(\ux), o)\\
& = B_{\zeta(\ux) } (b_0(\ux) , o )  + B_{\zeta (\ux) } (o, s^{-1} o) - B_{\zeta (\sigma \ux) } (b_0 (\sigma \ux), o),
\end{align*}
where we used that $s \zeta (\ux) = \zeta (\sigma \ux)$. 
Setting $v(\uy) = B_{\zeta(\uy)}(b_0(\uy),o)$ and $L' (\ux) = B_{\zeta (\ux) } (o, s^{-1} o)$, where $s$ is defined as above, the function $L$ and $L'$ are essentially cohomologous. The function $v $ is H\"older-continuous on each $p^{-1}(W^s_{loc})(\ux)$ by the same reasoning as the one for $L$ (in the beginning of Section~\ref{sec:surface}).
\end{proof}

By Lemma~\ref{lem:bowen}, $\beta_{\zeta(\uy)}(o, s^{-1}_{x_1}o)=L(\uy)+v(\uy)-v(\sigma \uy).$ Therefore,
\[ 
\frac{dm_{\sigma \ux}}{d \widehat{\sigma}_{*} m_x} (p^{-1}(\uy)) = q(x_1) e^{-h_{\Vol}\{ L(\uy) + v(\uy) - v(\sigma \uy)\}}.
\]

On the other hand, we know that the family of Lebesgue measures  $\{ \lambda_{\ux} \}$ on $W^s_{loc}$ is the unique family of measures (up to a global constant) which satisfies
\[ 
\frac{d\lambda_{\sigma \ux} }{ d\widehat{\sigma}_{*} \lambda_{\ux}}(p^{-1}(\uy)) = e^{-(n-1) L(\uy)}
\quad \quad \quad (***)
\]
Assume that the Bowen-Margulis measure projects on the Liouville measure. 

Then there is a positive function $w$ such that $\lambda_{\ux} =w m_{\ux}$, $m_L$-almost everywhere. 
For $m_L$-almost every $\ux$, the function $w$ has the following properties: 
\begin{align*}
\int_{W^s_{loc}(\ux)} w dm_{\ux} &=\lambda_{x}(W^s_{loc}(\ux)) \quad \;\mathrm{and}\\
\frac{w(\sigma^{-1} \uy)}{w(\uy)} = q(y_1)& e^{-h_{\Vol}\{ L(\uy) +v(\uy)-v(\sigma \uy)\}} e^{(n-1)L(\uy)}, 
\end{align*}
for $\lambda_{\ux}$-almost every $\uy$.
The second property comes from rearranging the following equality :
\[
\frac{d\lambda_{\sigma \ux}}{d\widehat{\sigma}_{*}(\lambda_{\ux})} =
\frac{w}{w \circ \sigma^{-1}} \frac{dm_{\sigma \ux}}{d\widehat{\sigma}_{*}  m_{\ux}}.
\]
Let us define the function $\Omega$ for $\uy, \uz$ in $W^{u}_{loc}(x)$ ,
\begin{align*}
\Omega (\uz,\uy) &= \underset{i=1, \cdots, \infty}{\prod}  \frac{q(z_{i+1})}{q(y_{i+1})}
\frac{e^{-h_{\Vol} \{L(\sigma^{i} \uz)+v(\sigma^{i}\uz) -v(\sigma^{i+1}\uz)\}}}{e^{-h_{\Vol} \{L(\sigma^{i} \uy)+v(\sigma^{i}\uy) -v(\sigma^{i+1}\uy)\}}}
\frac{e^{(n-1) L(\sigma^iz)}}{e^{(n-1) L(\sigma^iy)}}\\
&= \frac{e^{-h_{\Vol}v(\sigma \uz)}}{e^{-h_{\Vol}v(\sigma \uy)}}
 \underset{i=1, \cdots, \infty}{\prod}  
\frac{e^{-(h_{\Vol}-(n-1))L(\sigma^{i} \uz)}}{e^{-(h_{\Vol}-(n-1)) L(\sigma^{i} \uy)}}.
\end{align*}
It follows that
$$\frac{\Omega(\uy, \uz)}{\int \Omega(\uy, \uz)dm_{\ux}(\uz)}dm_{\ux}(\uy)$$ satisfies the equation $(***)$, thus is proportional to Lebesgue measure $\lambda_{\ux}$.
Therefore,
\[
\Omega(\uz,\uy) = \frac{w(\uz)}{w(\uy)}
\]
Since $v$ and $L$ are H\"older-continuous on $\Sigma$, the function $\Omega(\uz, \uy)$ is H\"older-continuous on $W^s_{loc}(\ux)$. Thus the function $\log w$ is H\"older-continous on each $W^{s}_{loc}(\ux)$.

We showed that 
$\log Q(\ux) -(h_{\Vol}-(n-1))L(\ux) = u(\ux) - u(\sigma \ux)$, for $u=\log w -h_{\Vol} v- \log q(x_0)$, which is H\"older-continuous on each $W^s_{loc}(\ux)$.

Similarly, using reversing time, we have $\log Q(\ux) -(h_{\Vol}-(n-1))L(\ux) = u'(\ux) - u'(\sigma \ux)$, for a function $u'$, which is H\"older-continuous on each $W^u_{loc}(\ux)$. By ergodicity of Liouville measure, the functions $u$ and $u'$ coincide almost everywhere (up to a constant).
\end{proof}

{\it Acknowledgement.} We are grateful to F. Paulin and J.-F. Lafont for helpful discussions. We would also like to thank the anonymous referee for invaluable remarks. The first author was supported in part by NSF Grant DMS-0801127.

\small{\end{document}